\documentclass[11pt, a4paper,leqno]{amsart}
\usepackage{amsmath,amsthm,amscd,amssymb,amsfonts, amsbsy}
\usepackage{latexsym}
\usepackage{txfonts}
\usepackage{exscale}

\usepackage[colorlinks,citecolor=red,pagebackref,hypertexnames=false]{hyperref}

\usepackage{pgf}
\usepackage{color}

\calclayout
\allowdisplaybreaks

\theoremstyle{plain}
\newtheorem{theorem}[equation]{Theorem}

\newtheorem{corollary}[equation]{Corollary}

\theoremstyle{definition}
\newtheorem{definition}[equation]{Definition}

\theoremstyle{remark}

\newcommand{\eps}{\varepsilon}

\newcommand{\re}{\mathbb{R}}
\newcommand{\rn}{\mathbb{R}^n}

\newcommand{\ree}{\mathbb{R}^{n+1}}

\newcommand{\dd}{\mathbb{D}}

\newcommand{\vp}{\varphi}

\newcommand{\xb}{{\bf x}}
\newcommand{\yb}{{\bf y}}

\newcommand{\Xb}{{\bf X}}
\newcommand{\Yb}{{\bf Y}}

\newcommand{\te}{\mathcal{T}_{\varepsilon}}

\begin{document}
\allowdisplaybreaks

\title[Nonhomogeneous Parabolic SIO]{Parabolic Singular Integrals with Nonhomogeneous Kernels}

\author{S. Bortz}
\address{Department of Mathematics
\\
University of Alabama
\\ 
Tuscaloosa, AL, 35487, USA}
\email{sbortz@ua.edu}
\author{J. Hoffman}
\address{Department of Mathematics
\\
University of Missouri
\\
Columbia, MO 65211, USA}
\email{jlh82b@mail.missouri.edu}
\author{S. Hofmann}
\address{
Department of Mathematics
\\
University of Missouri
\\
Columbia, MO 65211, USA} 
\email{hofmanns@missouri.edu}
\author{J.L.~Luna-Garcia}
\address{Department of Mathematics
\\
University of Missouri
\\
Columbia, MO 65211, USA}
\email{jlwwc@mail.missouri.edu}
\author{K. Nystr\"om}
\address{Department of Mathematics, Uppsala University, S-751 06 Uppsala, Sweden}
\email{kaj.nystrom@math.uu.se}

\thanks{S.B. was supported by the Simons foundation grant ``Travel support for Mathematicians” (grant number 959861). The authors J.H., S.H., and J.L.~L-G.  were 
partially supported by NSF grant  DMS-2000048.  
S.H. is currently supported by NSF grant DMS-2349846.
K.N. was partially supported by grant  2022-03106 from the Swedish research council (VR)}


\date{\today}



\maketitle

\begin{abstract}
We establish $L^2$ boundedness of all ``nice" parabolic singular integrals on ``Good Parabolic Graphs",
aka {\em regular} Lip(1,1/2) graphs.  The novelty here is that we include non-homogeneous kernels,
which are relevant to the theory of parabolic uniform rectifiability.  Previously, the third 
named author had treated the case of homogeneous kernels.  
The present proof combines the methods of that work (which in turn was
based on methods described in Christ's CBMS lecture notes), with the techniques of Coifman-David-Meyer.
\end{abstract}

\section{Introduction}
In this note we establish $L^2$ boundedness of ``nice" parabolic singular integrals on ``Good Parabolic Graphs",
aka {\em regular} Lip(1,1/2) graphs (see Theorem \ref{tsio}). In fact, combining Theorem \ref{tsio} and the results of \cite{BHHLN1}, we deduce 
that ``nice" parabolic singular integrals are $L^2$ bounded on parabolic uniformly rectifiable sets (Corollary \ref{csio})\footnote{In fact, our Corollary \ref{csio}
was stated previously as \cite[Corollary 4.9]{BHHLN1}, but the latter result 
 relies crucially on our results here, which are 
 quoted without proof in \cite{BHHLN1}.}. 
Thus, our work is the parabolic analogue of one direction of the characterization of classical uniform rectifiability, in terms of $L^2$ boundedness of
(elliptic) singular integrals, established in
\cite{DS1}.  Our approach combines 
the methods of \cite{H} with those of \cite{CDM}. 
At the time this note was initially written, our intention was eventually to characterize parabolic uniform rectifiability by the boundedness of parabolic SIOs, but that is a considerably more complicated task than we had first expected, for reasons we describe below.

The study of singular integral operators on rough sets was initiated by Calder\'on's work on the Cauchy integral \cite{C}, where he proved $L^2$ 
boundedness of the Cauchy integral operator on Lipschitz graphs with small constant. It was, however, the work of Coifman, McIntosh and Meyer \cite{CMM} that revolutionized the study of singular integrals on rough sets, 
and the developments that followed \cite{CDM, DJ-T1, DS1, DS2} 
would elucidate the deep connections between (quantitative) rectifiability and 
singular integrals, and more generally,   
between partial differential equations, operator theory, geometric measure theory and harmonic analysis.  A thorough literature review of these topics is not possible in this short introduction, but we refer the interested reader to \cite{BHMN} for a 
more extensive discussion of related history, pertaining to both PDE and singular integral theory.

As stated above, 
 the present result, on 
parabolic singular integrals with non-homogeneous kernels, is a first step towards
the goal of replicating, in a parabolic context, 
some of the elliptic theory of David and Semmes \cite{DS1, DS2}. 
In those works, David and Semmes show that $L^2$ 
boundedness of all ``nice" (elliptic) singular integral operators on an Ahlfors regular set implies that the set is uniformly (i.e., quantitatively) rectifiable (UR). The parabolic theory, however, turns out not to be as nice as
its elliptic analogue. 
Indeed, in the parabolic setting, it is known that the parabolic Lipschitz condition (``Lip(1,1/2)") is not sufficient for many desirable properties, 
particularly in partial differential equations and singular integral theory 
(see e.g. \cite{BHMN} for a historical overview). Here one must enforce additional regularity on the graph in the form of control on the half-order time derivative and this non-local condition causes some of the characterizations in \cite{DS2} to fail\footnote{See e.g. \cite[Observation 4.10]{BHHLN1}.}.

It is also true that, as stated, Corollary \ref{csio} can have no (unrestricted)
converse, that is, the boundeness of ``nice" parabolic singular integrals on a 
parabolic Ahlfors regular set does not (without further conditions)
imply that the set is parabolic uniformly rectifiable. Indeed, it was pointed out to the authors by B. Jaye that there are parabolic ADR sets which have parabolic SIO bounds, but are not parabolic UR. Indeed, in this setting, the time and space variables are decoupled in the sense that we require oddness only in $x$. This allows one to build a parabolic Ahlfors regular set which is a union of purely spatial planes, 
and has the property that $T1 = 0$ for all (convolution)
kernels that are odd in the space variable. 

Recently, B. Jaye and the second named author have proved a suitable (partial) converse to Corollary \ref{csio}. They show that the example described above is
in some sense the only enemy. In short, they introduce coefficients that measure how close a set is to being a union of spatial planes in transportation distance 
(that is, a modification of the alpha numbers of Tolsa \cite{T}), 
and show that if nice parabolic SIOs are $L^2$-bounded, and in addition,
 if for every $\epsilon > 0$, the set where these alpha-coefficients is {\it less} than 
 $\epsilon$ is a Carleson set, then the set is parabolic uniformly rectifiable \cite{HoffJa}.

\section{Notation, definitions, and statement of results}

\noindent{\bf Notation:}
Our ambient space is $(n+1)$-dimensional space-time:
\[\ree = \big\{ \Xb:=(X,t) = (x_0,x,t) \in \re\times \re^{n-1}\times\re\big\}\,,\]
and we shall also at times work with $n$-dimensional space-time
\[\rn = \big\{ \xb:= (x,t)  \in \re^{n-1}\times\re\big\}\,.\]
We write $(\xi,\tau) \in \re^{n-1}\times \re$ to denote points on the Fourier transform side of
space-time.
We also write $\rn_{sp}$ to denote purely spatial $n$-dimensional Euclidean space.
Note in particular that we shall often distinguish one spatial variable in $\rn_{sp}$, and write, e.g., 
$X=(x_0,x) \in \re\times \re^{n-1}$, and that we use lower case letters
$x,y,z$ to denote spatial points in $\re^{n-1}$, and 
capital letters $X=(x_0,x),Y=(y_0,y),Z=(z_0,z)$ to denote points in  $\rn_{sp}=\re\times \re^{n-1}$.
In accordance with the notation introduced above,
we let $\|\Xb\|=\|(X,t)\|= \|(x_0,x,t)\|$ and $\|\xb\|=\|(x,t)\|$ denote the parabolic length 
of the vectors $\Xb=(X,t)\in\ree$ and $\xb=(x,t)\in\rn$, respectively, that is, 
\[\|\Xb\|=\|(X,t)\| = |X| + |t|^{1/2}, \quad \|\xb\|=\|(x,t)\ = |x| + |t|^{1/2}.\]

We let $d=n+1$ denote the parabolic homogeneous dimension of space-time $\rn$.

We define a fractional integral operator $I_p$ of parabolic order 1 on $\rn$
by means of the Fourier transform:
\[\widehat{I_{p} f}(\xi,\tau):= \|(\xi,\tau)\|^{-1} \widehat{f} (\xi,\tau)\,.
\]

\begin{definition}[\em Parabolic Calder\'on-Zygmund kernels]\label{defczn} For an integer $N\geq 1$,
we shall say that a kernel $K=K(X,t)$ satisfies a ($d$-dimensional)
parabolic C-Z(N) condition, and we write $K\in C\text{-}Z(N)$,
if it satisfies the following properties:
\begin{enumerate}
\item[(i)] Smoothness with C-Z estimates:  $K \in C^N(\ree\setminus \{0\})$, with the estimate
\[
|\nabla_X^j\,\partial_t^kK(X,t)| \,\leq\, C_{j,k}\, \|(X,t)\|^{-d-j-2k} \,, \quad \forall \, 0\leq j+k\leq N\,.
\]
\item[(ii)] Oddness in {\em spatial} variables:  $K(X,t) = -K(-X,t)$, for each $(X,t) \in \rn_{sp}\times\re$.
\end{enumerate}
We shall simply say that $K$ is a C-Z kernel (and we write $K\in C\text{-}Z$), if
$K\in C\text{-}Z(N)$ for some positive integer $N$.
We shall also consider analogous kernels $H(x,t)$ defined on $\rn$, {\em not} necessarily odd in the space
variables, but still
satisfying the $n$-dimensional version of
property (i) above, i.e., 
\begin{equation}\label{eqczn}
|\nabla_x^j\,\partial_t^k H(x,t)| \,\leq\, C_{j,k}\, \|(x,t)\|^{-d-j-2k} \,, \quad \forall \, 0\leq j+k\leq N\,.
\end{equation}
In this case we shall say that $H$ satisfies the $C\text{-}Z(N)$(i) condition, or
$H\in C\text{-}Z(N)$(i).
\end{definition}
Note that we do not assume homogeneity of $K$ and $H$ in the preceding definition.

\begin{definition}[\em Regular Lip(1,1/2) functions and Good Parabolic Graphs]\label{defreg} We say that a function
$A:\rn \to \re$ is a Lip(1,1/2) function if there exists a constant $L \ge 0$ such that
\[|A(\xb) - A(\yb)| = |A(x,t) - A(y,s)| \le L(|x - y| +  |t - s|^{1/2}), \quad \forall \xb, \yb \in \mathbb{R}^n.\]
We say a function
$A:\rn \to \re$ is a regular Lip(1,1/2) function if it is Lip(1,1/2), and if in addition $\dd_n A\in BMO$,
where $\dd_n:= I_{p} \circ \partial_t$ is the half-order time derivative. Following \cite{H}, we endow the
regular Lip(1,1/2) functions with the norm
\[\|A\|_{comm} := \|\nabla_x A\|_{L^\infty(\rn)} + \|\dd_n A\|_{BMO(\rn)}\,.\]
Of course, BMO is parabolic BMO, defined with respect to parabolic cubes (or balls).

We say that a graph $\Gamma := \left\{\big(A(\xb),\xb\big)\right\}$ is a Good Parabolic Graph, and we write
 $\Gamma \in GPG$, if $A$ is a regular Lip(1,1/2) function.
\end{definition}

``Surface measure\footnote{See Appendix B in \cite{BHHLN2}.}"  on $\Gamma$ is defined to be $d\sigma(x,t):= \sqrt{1+|\nabla_x A(x,t)|^2} dx dt$.

Given a Good Parabolic Graph $\Gamma$, and a C-Z kernel $K$, we define
associated truncated singular integral operators as follows, for each $\eps>0$:
\[T^\Gamma_\eps f(\Xb):=\int_{\Gamma\cap \|\Xb-\Yb\|>\eps} K\big(\Xb-\Yb\big)\, f(\Yb)\, d\sigma(\Yb)\,,
\quad \Xb\in \Gamma\,,
\]
and using graph co-ordinates $\Gamma:=
\left\{\big(A(\xb),\xb\big)\right\}$, we write the corresponding Euclidean version
\begin{equation}\label{eqTdef}
T_\eps f(\xb):=\int_{\|\xb-\yb\|>\eps} K\big(A(\xb)-A(\yb),\xb-\yb\big)\, f(\yb)\, d\yb\,,\quad \xb\in \rn\,.
\end{equation}

We have the following.
\begin{theorem}\label{tsio} Let $K\in C\text{-}Z(2N)$, and suppose that $\Gamma
:= \left\{\big(A(\xb),\xb\big)\right\} \in GPG$.
If $N$ is large enough, then for some universal constant $N_1$,
we have the uniform $L^2$ bound
\begin{equation*}
\sup_{\eps>0} \|T^\Gamma_\eps f\|_{L^2(\Gamma)} \, \leq\, C(1+\|A\|_{comm})^{N_1}\| f\|_{L^2(\Gamma)}\,, 
\end{equation*}
where $C$ depends on $K$ and $n$. 
\end{theorem}

The case that $K$ is {\em homogeneous} was previously treated in \cite{H2}; some particular homogeneous
kernels arising in the theory of parabolic layer potentials were originally treated by Lewis and Murray \cite{LM}.

Combining Theorem \ref{tsio} with the results of \cite{BHHLN1}, and using the ``big pieces/
good-$\lambda$" method of G. David (see \cite[Proposition III.3.2]{D}), 
we obtain as an immediate corollary the following
(see \cite[Section 4]{BHHLN1} for more details):

\begin{corollary}\label{csio}
The conclusion of Theorem \ref{tsio} continues to hold when the graph $\Gamma$ is replaced by a
parabolic uniformly rectifiable set.
\end{corollary}

In order to prove Theorem \ref{tsio}, we will need the following pair of results.

\begin{theorem}\label{convolution} Let $H$ be a kernel defined on $\rn\setminus \{0\}$,
odd in the space variable, and
satisfying \eqref{eqczn} with $N=1$, i.e., 
$H\in C\text{-}Z(1)$(i).  Define truncated convolution singular integral operators
\[
S_\eps f(\xb):=\int_{\|\xb-\yb\|>\eps} H(\xb-\yb)\, f(\yb)\, d\yb\,,\quad \xb\in \rn\,.
\]
Then $S_\eps:L^2(\rn) \to L^2(\rn)$ uniformly in $\eps$, i.e.,
\[\sup_{\eps>0} \|S_\eps f\|_{L^2(\rn)} \leq \, C\|f\|_{L^2(\rn)}
\]
\end{theorem}

The theorem is well known, and we omit the proof.  One can either estimate the Fourier transform
of the truncated kernel, and use Plancherel's theorem, or else use the parabolic $T1$ theorem.

\begin{theorem}\label{commutator}
Let $H$ be a kernel 
defined on $\rn\setminus \{0\}$, even in the space variable, and
satisfying \eqref{eqczn} with $N=1$, i.e., 
$H\in C\text{-}Z(1)$(i).  Let $A$ be a regular Lip(1,1/2) function.
Define  the truncated first commutator 
\[
\mathcal{C}_\eps f(\xb):=\int_{\|\xb-\yb\|>\eps} \frac{A(\xb)-A(\yb)}{\|\xb-\yb\|}\,
H(\xb-\yb)\, f(\yb)\, d\yb\,,\quad \xb\in \rn\,.
\]
Then $\mathcal{C}_\eps:L^2(\rn) \to L^2(\rn)$ uniformly in $\eps$, i.e.,
\[\sup_{\eps>0} \|\mathcal{C}_\eps f\|_{L^2(\rn)} \leq \, C\|A\|_{comm}\,\|f\|_{L^2(\rn)}
\]
\end{theorem}
In the case that $H$ is homogeneous, the result is proved in \cite{H}. We shall
prove the general case here.  We defer the proof to the end of this note,
and take Theorem \ref{commutator} for granted for the moment.

\section{Proofs of the Theorems}

We take Theorem \ref{commutator} for granted until the end of this section.

\begin{proof}[Proof of Theorem \ref{tsio}]
Note that it is equivalent to prove
\begin{equation}\label{eql2}
\sup_{\eps>0} \|T_\eps f\|_{L^2(\rn)} \, \leq\, C(1+\|A\|_{comm})^{N_0}\| f\|_{L^2(\rn)}\,, 
\end{equation}
for some universal constant $N_0$,
where $T_\eps$ is defined as in \eqref{eqTdef}. Indeed, the difference is bounded by a (parabolic) maximal function.
To this end, we first use the technique of \cite{CDM} to reduce matters to 
the case of ``singular integrals of Calder\'on type".  The method is nowadays well-known, but 
for the reader's convenience, we provide the details.
Observe that 
\[ 
K(x_0,\xb) = K\left(\|\xb\|\,\frac{x_0}{\|\xb\|}\,,\xb\right)\,=:
\,K_{\xb}\left(\frac{x_0}{\|\xb\|}\right)\,,
\]
where $K_{\xb}(\kappa):= K\big(\|\xb\|\kappa,\xb\big)$.
Consequently,
\[
K(x_0,\xb) = \int_{\re}e^{2\pi i \frac{x_0}{\|\xb\|} \zeta} \,\,
\widehat{K_{\xb}}\,(\zeta)\, 
d\zeta
\,=:\, \int_{\re}e^{2\pi i \frac{x_0}{\|\xb\|} \zeta} \, H_\zeta(\xb)\, d\zeta\,,
\]
where the Fourier transform acts in one dimension, i.e.,
\[H_\zeta(\xb) := \int_{\re} e^{-2\pi i \zeta \kappa}\, K\big(\|\xb\|\kappa,\xb\big)\, d\kappa\,.
\]
Using the $C\text{-}Z(2N)$ condition, we find that for $0\leq j+k\leq N$, and $0\leq m\leq N$,
\begin{multline*}\|(x,t)\|^{d+j+2k}\,\left | \zeta^m \nabla_x^j\,\partial_t^k \, H_\zeta(\xb)\right| \\[4pt]
=\,C_m \|(x,t)\|^{d+j+2k}\left | \int_{\re} e^{-2\pi i \zeta \kappa}\, 
\partial_\kappa^m \nabla_x^j\,\partial_t^k K\big(\|\xb\|\kappa,\xb\big)\, d\kappa\right| \\[4pt]
\leq C_{m,j,k} \int_{\re}(1+\kappa)^{-d} \, d\kappa \leq C_{m,j,k}\,,
\end{multline*}
i.e.,  
$(1+|\zeta|)^N H_\zeta(\xb)$ satisfies \eqref{eqczn}, 
for each $\zeta \in \re$, with constants that are uniform in $\zeta$.

Next, we split $H_\zeta(\cdot, t)$ into its (spatial) even and odd parts, i.e.,
$H_\zeta = H_\zeta^{even} + H_\zeta^{odd}$, where
\[H_\zeta^{even}(x,t) =\frac12\big(H_\zeta(x,t) + H_\zeta(-x,t)\big)\,,\quad 
H_\zeta^{odd}(x,t) =\frac12\big(H_\zeta(x,t) - H_\zeta(-x,t)\big)\,.
\]
We now claim that $H_\zeta^{even}$ and $H_\zeta^{odd}$ have opposite parity, 
when viewed as functions of $\zeta$, i.e.,
\begin{equation}\label{eqparity}
H_\zeta^{even}(x,t) = - H_{-\zeta}^{even}(x,t)\,,\quad H_\zeta^{odd}(x,t) =  H_{-\zeta}^{odd}(x,t)\,.
\end{equation}
To see this, observe first that
\begin{multline*}
H_{-\zeta}(\xb) := \int_{\re} e^{-2\pi i (-\zeta) \kappa}\, K\big(\|\xb\|\kappa,\xb\big)\, d\kappa\,
=\, \int_{\re} e^{-2\pi i \zeta \kappa}\, K\big(-\|\xb\|\kappa,\xb\big)\, d\kappa \\[4pt]
=\, -\int_{\re} e^{-2\pi i \zeta \kappa}\, K\big(\|\xb\|\kappa,-x,t\big)\, d\kappa
\,,
\end{multline*}
where we have made the change of variable $\kappa \to -\kappa$, and then used oddness of $K$ in the space 
variables.  Taking, for example, the even part in $x$, we have
\[
H_{-\zeta}^{even}(x,t) = -\frac12 \int_{\re} e^{-2\pi i \zeta \kappa}\, \left(K\big(\|\xb\|\kappa,-x,t\big)
+ K\big(\|\xb\|\kappa,x,t\big)\right) \, d\kappa = -H_\zeta^{even}(x,t)\,,
\]
as claimed.  The odd part $H_\zeta^{odd}$ may be treated by a similar argument.  We omit the details.
Thus \eqref{eqparity} holds.

Taking even and odd parts, and using  \eqref{eqparity}, we see that
\begin{equation}\label{eqrep}
K(x_0,\xb) = \int_{\re}\sin\left(2\pi  \frac{x_0}{\|\xb\|} \zeta\right)  \, H^{even}_\zeta(\xb)\, d\zeta
\,+\, \int_{\re}\cos\left(2\pi  \frac{x_0}{\|\xb\|} \zeta\right)  \, H^{odd}_\zeta(\xb)\, d\zeta\,.
\end{equation}
To conclude the proof of Theorem \ref{tsio}, it is therefore enough to prove the following.
\begin{theorem}\label{tsio2} Let $H$ be a singular kernel satisfying \eqref{eqczn}, 
with $N=1$.
Define truncated singular integrals of ``Calder\'on-type" by
\begin{equation}\label{calddef}
\mathcal{T}_\eps f(\xb):=\int_{\|\xb-\yb\|>\eps} E\left(\frac{A(\xb)-A(\yb)}{\|\xb-\yb\|}\right)\,
H(\xb-\yb)\, f(\yb)\, d\yb\,,\quad \xb\in \rn\,,
\end{equation}
where either $H(x,t)$ is odd in $x$, and $E=\cos$, or $H(x,t)$ is even in $x$, and $E=\sin$.

Then there is a universal constant $N_0$ such that 
\[\sup_{\eps>0} \|\mathcal{T}_\eps f\|_{L^2(\rn)} \leq \, C(1+\|A\|_{comm})^{N_0}\|f\|_{L^2(\rn)}\,,
\]
where $C$ depends on $n$ and the constants in \eqref{eqczn}.  
\end{theorem}

Let us take Theorem \ref{tsio2} for granted momentarily. Setting
$x_0 = A(\xb)-A(\yb)$ in \eqref{eqrep},
we see that 
the truncated singular integral $T_\eps$ defined in \eqref{eqTdef} can be represented as a sum of two terms
of the form
\[\int_{\re} \mathcal{T}_\eps^\zeta f \, d\zeta\,,
\]
where $\mathcal{T}_\eps^\zeta$ is defined as in \eqref{calddef}, with $A$ replaced by
$\zeta A$, and either with $E=\cos$ and $H =H^{odd}_\zeta$, or with $E=\sin$ and $H =H^{even}_\zeta$.
Using our previous observation that
$(1+|\zeta|)^N H_\zeta$ satisfies the bounds in \eqref{eqczn}, uniformly in $\zeta$, 
and invoking Theorem \ref{tsio2}, we find that
\begin{multline*}
\sup_{\eps>0} \|T_\eps f\|_{L^2(\rn)} \, \lesssim \int_{\re}(1+|\zeta|)^{-N}(1+|\zeta|\,\|A\|_{comm})^{N_0} \, 
d\zeta\,  \| f\|_{L^2(\rn)} \\[4pt]
\lesssim (1+\|A\|_{comm})^{N_0} \,\| f\|_{L^2(\rn)}\,, 
\end{multline*}
provided that $N >N_0+1$.  Thus \eqref{eql2} holds, and the conclusion of Theorem \ref{tsio} follows.

To complete the proof of Theorem \ref{tsio}, it therefore remains to prove
Theorem \ref{tsio2}.

\begin{proof}[Proof of Theorem \ref{tsio2}]
The theorem was already proved in \cite{H2}, in the case that $H$ is homogeneous.  
The present proof is a modified version of the argument in \cite{H2}, which in turn follows
from the methods in \cite{Ch}.  For specificity,
we treat the case that $E=\cos$, and $H(x,t)$ is odd in $x$, for each fixed $t$; the proof in the other case is
essentially the same.  To simplify notation, we set
\[M:=\|A\|_{comm}\,.\] 
It is enough to verify the localized $T1$ estimate
\begin{equation}\label{eqT1}
\fint_B |\te \eta_B| \leq C(1+M)^{N_0}
\end{equation}
(and the same for the transpose of $\te$, but the latter is of the same form as
$\te$), uniformly in $\eps>0$ and in every parabolic ball 
\[ B=B_r(\xb_B):=\{\xb \in \rn:\, \|\xb-\xb_B\|<r\}\,,\]
where $\eta_B \in C_0^\infty(5B)$, with $\eta_B\equiv 1$ on $4B$, and $0\leq \eta_B\leq 1$.

By scale-invariance\footnote{Even though our 
individual kernels
are not homogeneous, they do form a scale-invariant class, with scale-invariant control of the
various relevant constants.}, we may suppose that the ball $B$  in \eqref{eqT1}
has radius $r=1$.

Following \cite{Ch} , we choose 
$\varphi \in C_0^\infty (1/4,1)$ such that 
\begin{equation}\label{phidef}
\int\limits_0^\infty \varphi (\rho )\,\frac{d\rho }\rho =1,
\end{equation}
We may then replace $\te \eta_B(\xb)$ by 
\begin{equation}\label{eqteother}
 \int_\eps^1\!\int_{\rn} \vp\left(\frac{\|\xb-\yb\|}{\delta} \right)
\cos\left(\frac{A(\xb)-A(\yb)}{\|\xb-\yb\|}\right)\,
H(\xb-\yb)\,  d\yb\,\frac{d\delta}{\delta}\,,
\end{equation}
since for $\xb \in B$,
the error is bounded by a uniform constant depending only on $n$ and $H$.
For $\delta>0$, we define a nice parabolic approximate identity
$P_\delta f:= \Phi_\delta * f$, where
$\Phi \in C_0^\infty(B_1(0))$, $0\leq \Phi$, $\int\Phi = 1$, and 
\[\Phi_\delta(x,t):= \delta^{-d} \Phi\left(\delta^{-1} x,\delta^{-2} t\right)\,.
\]
We then write 
\begin{multline*}
\cos \left(\frac{ A(\xb)-
A(\yb)}{\parallel \xb-\yb\parallel}\right)\,
= \, \cos \left( \frac{(x-y)\cdot P_\delta \nabla
_x A(\xb)}{\parallel \xb-\yb\parallel } \right)  \\[4pt] +\,\,
\frac{ A(\xb)- A(\yb)-(x-y)\cdot
P_\delta \nabla _x A(\xb)}{\parallel \xb-\yb\parallel }\,\sin \left( \frac{
(x-y)\cdot P_\delta \nabla _x A(\xb)}{\parallel
\xb-\yb\parallel }\right) \\[4pt]
+\,\, O\left( \frac{\mid  A(\xb)- A(\yb)-(x-y)\cdot P_\delta \nabla _x A(\xb)\mid ^2}{\parallel \xb-\yb\parallel ^2}
\right) \\[4pt]=: \,I(\xb,\yb) + I I(\xb,\yb) + III(\xb,\yb)
\end{multline*}
Since $H(x,t)$ is odd in $x$, term $I(\xb,\yb)$ contributes zero to the integral in \eqref{eqteother}, 
as does the part of term $II(\xb,\yb)$ involving $(x-y)\cdot P_\delta \nabla _x A(\xb)$ in the numerator 
of the first factor; thus, the contribution of term $II(\xb,\yb)$ equals

\begin{multline}\label{iipart} 
\int_\eps^1\!\int_{\rn} \vp\left(\frac{\|\xb-\yb\|}{\delta} \right)
\frac{A(\xb)-A(\yb)}{\|\xb-\yb\|}
\sin \left( \frac{
(x-y)\cdot P_\delta \nabla _x A(\xb)}{\parallel
\xb-\yb\parallel }\right)
H(\xb-\yb)\,  d\yb\,\frac{d\delta}{\delta}\\[4pt]
= \sum_{j,k} a_{k,j}\big(P_\delta \nabla _x A(\xb)\big)
\int_\eps^1\!\int_{\rn} \vp\left(\frac{\|\xb-\yb\|}{\delta} \right)
\frac{A(\xb)-A(\yb)}{\|\xb-\yb\|}
\widetilde{H}_{k,j}(\xb-\yb)\,  d\yb\,\frac{d\delta}{\delta}\,,
\end{multline}
where $\widetilde{H}_{k,j}= Y_{k,j} H$, and with $\vec{b} = P_\delta \nabla _x A(\xb)$,
\[\sum_{j,k} a_{k,j}(\vec{b}) \,Y_{k,j} \left(\frac{\yb}{\|\yb\|}\right) = \sin \left( \frac{
y\cdot \vec{b}}{\parallel
\yb\parallel }\right)=:f(\vec{b},\yb)
\]
is the orthonormal spherical harmonic expansion (with $k$ being the degree of $Y_{k,j}$) of the
function $f(\vec{b},\yb):= \sin(\|\yb\|^{-1} y\cdot\vec{b})$, which is homogeneous of degree zero
(with respect to parabolic dilations), and thus is determined by its values on the sphere.
Note that by oddness of the Sine function, the $Y_{k,j}$ may be
chosen to be odd in the spatial variable (since the even and odd parts of a spherical harmonic
are also spherical harmonics), 
and therefore, since $H$ has the same property, it follows that
$\widetilde{H}_{k,j}(x,t)= Y_{k,j}(x,t) H(x,t)$ is even in $x$.
Furthermore, for every non-negative integer $m$,  
\[ \sup_{|\vec{b}|\leq M }\, \sup_{\yb \in S^{n-1}} |\nabla_{y,s}^m f(\vec{b},\yb)| 
\leq C_m M^m
\]
so by a well known property of the spherical harmonics
(see, e.g., \cite[pp. 70-71]{S}) we have
\begin{equation}
\parallel a_{k,j}\parallel _{L^\infty (\mid \vec{b}\mid \leq M)} 
\,\leq \,C_m \,M ^{m}\, k^{-m} 
\end{equation}
with $m$ arbitrarily large.  Let us recall also that by standard facts for spherical harmonics
(see, e.g., \cite{CZ}), the dimension $h_k$ of the space of spherical harmonics of degree $k$, satisfies
\begin{equation}\label{hk}
h_k \leq \, C_n\, k^{n-2}\,,
\end{equation}
and also, the normalized spherical harmonics $Y_{k,j}$ satisfy
\begin{equation}\label{Yest}
\sup_{\yb \in S^{n-1}}  |Y_{k,j}(\yb)| \,+\,
\sup_{\yb \in S^{n-1}} k^{-1} |\nabla_{y,s} Y_{k,j}(\yb)|\, \leq \,C_n \, k^{(n-2)/2}\,.
\end{equation}
Let us further note that for $\xb \in B$, the integral in \eqref{iipart} is unchanged if
we insert the factor $\eta_B(\yb)$, since the latter equals 1 in $4B$.
Combining these observations, we find that 
\eqref{iipart} is bounded in absolute value by
\[ F(\xb):=\, C_m M^m \sum_{k\geq 1} k^{-m +1+(n-2)/2}\, \sum_{j=1}^{h_k}
\left|\mathcal{C}^{k,j}_\eps\,\eta_B(\xb)\right|\,,
\]
for $\xb \in B$, where $\mathcal{C}^{k,j}_\eps$ is a smoothly truncated
version of a parabolic Calder\'on-type commutator, defined as in Theorem \ref{commutator}, with respect to the 
normalized kernel $k^{-1-(n-2)/2} \widetilde{H}_{k,j}$, which as noted above is even in $x$.
By this normalization, \eqref{Yest}, and Theorem \ref{commutator}, we have
\[ \sup_{j,k} \sup_{\eps>0} \|\mathcal{C}^{k,j}_\eps\|_{L^2(\rn)\to L^2(\rn)} \leq \,CM\,.
\]
Consequently, by \eqref{hk}, we have
\[\fint_B|F(\xb)| d\xb\leq \left(\fint_B|F(\xb)|^2 d\xb\right)^{1/2} 
 \lesssim M^{m+1}\,, 
\]
provided that we fix $m> 1+3(n-2)/2$.  For the contribution of term $II(\xb,\yb)$, we therefore
obtain the desired bound in \eqref{eqT1}, with $N_0=m+1$.

Finally , we consider term $III(\xb,\yb)$.  Define
\[\gamma_A(\xb,\delta) := \left(\delta^{-d-2}\int_{\delta/4\leq \|\xb-\yb\|\leq \delta}
 \mid  A(\xb)- A(\yb)-(x-y)\cdot P_\delta \nabla _x A(\xb)\mid ^2 \, d\yb\right)^{1/2}\,.
 \]
After plugging $III(\xb,\yb)$ into the integral in \eqref{eqteother},
and then averaging over the (unit) ball $B$, we obtain the bound
\[|B|^{-1} \int_B \int_0^1 \gamma^2_A(\xb,\delta)\, \frac{d\delta}{\delta} d\xb
 \leq \,C M^2\,,
\]
by the parabolic Dorronsoro-type result proved in \cite[p.249, estimate (35)]{H2}.
\end{proof}

With Theorem \ref{tsio2} in hand, this concludes the proof of  
Theorem \ref{tsio}.
\end{proof}

\begin{proof}[Proof of Theorem \ref{commutator}]
As above, it is enough to verify the localized $T1$ estimate
\eqref{eqT1}, but now with $\mathcal{C}_\eps$ in place of $\te$.
Once again by dilation invariance (within the class of kernels under consideration), we may suppose that
$B$ has radius $r=1$.
By the localization lemma in \cite[Appendix, Lemma 2]{H2}, we may assume that $A$ is supported in $10B$, 
and that $\nabla_x A$ and $\dd_nA$ belong to $L^2(\rn)$, with
\begin{equation}\label{local}
\left(|B|^{-1}\int_{\rn}\left( |\nabla_x A(\xb)|^2 + |\dd_n A(\xb)|^2\right)\, d\xb\right)^{1/2} \,\lesssim\, \|A\|_{comm} =: M\,. 
\end{equation}

For notational convenience, with $\xb = (x,t)$, we shall denote the parabolic dilations as follows:
\[ \rho^{(1,2)} \xb:= (\rho x,\rho^2 t)\,,\quad \rho>0\,.
\]
We also denote points on the unit sphere as $\omega= (\omega',\omega_n)$, where
$\omega_n$ is the component of $\omega$ in the time direction.
Using parabolic polar coordinates (see, e.g., \cite{FR1}), 
and integrating by parts, we may write
\begin{multline*}
C_\eps \eta_B (\xb) \\[4pt]
=\, \int_{S^{n-1}} \int_\eps^\infty \frac{A(\xb) - A\left(\xb - \rho^{(1,2)} \omega\right)}{\rho}
H\left(\rho^{(1,2)} \omega \right)\, \rho^{d-1} \, d\rho \, (1+\omega_n^2)\,d\sigma(\omega)
\\[4pt]
=\, \int_{S^{n-1}} \int_\eps^\infty \omega'\cdot \nabla_x A\big(\xb - \rho^{(1,2) }\omega\big)\,
  \widetilde{H}(\rho^{(1,2)}\omega)
\, \rho^{d-1}\,d\rho \, (1+\omega_n^2)\,d\sigma(\omega)
\\[4pt]
+\,\, \int_{S^{n-1}} \int_\eps^\infty 2\rho\omega_n \, \partial_t A\big(\xb - \rho^{(1,2) }\omega\big) \,
  \widetilde{H}(\rho^{(1,2)}\omega)
\, \rho^{d-1}\,d\rho \, (1+\omega_n^2)\,d\sigma(\omega) \\[4pt]
+\,  \, \text{boundary terms}\, =: \,I + II +\text{boundary terms},
\end{multline*}
where the boundary term at $\rho =\eps$ is uniformly bounded by $CM$, and where
\[  \widetilde{H}(\rho^{(1,2)}\omega) := \rho^{1-d} \int_\rho^\infty r^d H\left(r^{(1,2)} \omega \right)\, \frac{dr}{r^2} \,
=\, \int_1^\infty r^d H\left((r\rho)^{(1,2)} \omega \right)\, \frac{dr}{r^2}\,,
\]
i.e., with $\rho^{(1,2)}\omega =\xb-\yb$,
\[  \widetilde{H}(\xb-\yb) = \int_1^\infty r^d H\left(r^{(1,2)} (\xb-\yb) \right)\, \frac{dr}{r^2}
\]
Note that $\widetilde{H}$ satisfies \eqref{eqczn} with $N=1$, and $\widetilde{H}$ is even in the space
variable (since these properties holds for $H$, by 
assumption).
Returning to rectangular coordinates, we see that
\[ I = \int_{\|\xb-\yb\|>\eps} \frac{x-y}{\|\xb-\yb\|} \cdot \nabla_y A(\yb)\, \widetilde{H}(\xb-\yb)\,d\yb
=: S_\eps (\nabla A)(\xb)\,,
\]
where $S_\eps$ is the truncated convolution singular integral operator with the vector-valued 
C-Z kernel
$H_0(\xb):= x \|\xb\|^{-1}  \widetilde{H}(\xb)$, which is odd in $x$.  Thus, 
\[\fint_B |I| \, d\xb \leq \left(\fint_B|S_\eps (\nabla A)(\xb)|^2\, d\xb\right)\,\leq\, CM\,,
\]
by Theorem \ref{convolution} and \eqref{local}.

Finally, writing term $II$ in rectangular coordinates yields
\[ II= 2\int_{\|\xb-\yb\|>\eps}  \frac{t-s}{\|\xb-\yb\|}\,\partial_s A(\yb)\, \widetilde{H}(\xb-\yb)\,d\yb =:
J_\eps * \partial_s A (\xb)\,,
\]
where
\[J_\eps (\xb):=  \frac{t-s}{\|\xb-\yb\|}\, \widetilde{H}(\xb-\yb)\, 1_{\|\xb\|>\eps}\,.
\]
A routine computation shows that
\[|\widehat{J_\eps}(\xi,\tau)| \lesssim \|(\xi,\tau)\|^{-1}\,,
\]
uniformly in $\eps$, and hence that 
\[\|J_\eps * \partial_s A \|_{L^2(\rn)} \,\lesssim\, \|\dd_n A\|_{L^2(\rn)}\, \lesssim \,M\,,
\]
by  \eqref{local} and the fact that $B$ has radius 1.  Consequently,
\[\fint_B |II| \,d\xb \,\lesssim \,M\,.
\]
\end{proof}

\end{document}